\documentclass[a4paper,11pt,english,twoside,leqno]{amsart}

 \usepackage{etex} 
\usepackage[T1]{fontenc}
\usepackage[latin1]{inputenc}
\usepackage{pifont}
\usepackage{euscript}
\usepackage{mathrsfs}
\usepackage{stmaryrd}
\usepackage{lmodern}
\rmfamily
\DeclareFontShape{T1}{lmr}{b}{sc}{<->ssub*cmr/bx/sc}{}
\DeclareFontShape{T1}{lmr}{bx}{sc}{<->ssub*cmr/bx/sc}{}

\DeclareMathAlphabet{\mathpzc}{OT1}{pzc}{m}{it}

\usepackage[textwidth=14cm,heightrounded,hcentering]{geometry}
\usepackage[a4paper, pdfstartview=FitV, linkcolor=blue, citecolor=blue, urlcolor=blue]{hyperref}

\usepackage{multicol}
\usepackage{enumerate}
\usepackage{enumitem}
\setlist[enumerate,1]{label=(\roman*)}
\usepackage{calc}

\usepackage{graphicx,subfig} 
\usepackage{float} 
\graphicspath{{Figures/}}

\usepackage[all,2cell]{xy} \UseAllTwocells
\usepackage{tikz-cd}
\usepackage{amsthm,amssymb,amsmath}
\usepackage[fixamsmath]{mathtools} 
\usepackage{leftidx} 
\usepackage{cases}

\usepackage{xspace}

\newcommand{\cal}{\mathcal}

\newcommand{\Gq}{\Gal(\overline{\QQ}/\QQ)}

\newcommand{\NN}{{\mathbb N}}
\newcommand{\ZZ}{{\mathbb Z}}
\newcommand{\QQ}{{\mathbb Q}}

\newcommand{\CC}{{\mathbb C}}
\newcommand{\PP}{{\mathbb P}}

\renewcommand{\O}{{\cal O}}

\newcommand{\D}{{\cal D}}

\newcommand{\M}{{\cal M}}

\newcommand{\Spec}{{\operatorname{Spec}\kern 1pt}}
\newcommand{\Spf}{{\operatorname{Spf}\kern 1pt}}

\newcommand{\ord}{{\operatorname{ord}\kern 1pt}}

\renewcommand{\mod}{\mathrm{\,mod\,}}

\def\buildrel#1\over#2{\mathrel{\mathop{\kern 0pt#2}\limits^#1}}

\newcommand{\Ext}{{\mathrm{Ext}}}

\newcommand{\Hom}{{\mathrm{Hom}}}

\renewcommand{\H}{{\mathrm{H}}}

\renewcommand{\Im}{\operatorname{Im}\kern 1pt}
\newcommand{\coker}{\operatorname{coker}\kern 1pt}

\newcommand{\Gal}{{\mathrm{Gal}}}

\renewcommand{\epsilon}{\varepsilon}
\newcommand{\Aut}{{\mathrm{Aut}}}
\newcommand{\Frac}{\mathrm{Frac}}

\newcommand{\Ar}{{\mathbf{Ar}\kern 0.5pt}}

\newcommand{\cd}{\operatorname{cd}\kern 1pt}

\renewcommand{\tilde}{\widetilde}

\newcommand{\limind}{\mathop{\mathop\mathrm {lim}\limits_{\xrightarrow{\hskip 0.5cm}}}}

\newcommand{\cart}{\ar@{}[dr] |{\square}}
\newcommand{\comm}{\ar@{}[dr] |{\circlearrowleft}}
\newcommand{\incl}[1][r]{\ar@<-0.2pc>@{^(-}[#1] \ar@<+0.2pc>@{-}[#1]}

\numberwithin{equation}{section}

\newcounter{MyCount}[subsection]
\renewcommand{\theMyCount}{\thesubsection.\arabic{MyCount}}
\newcommand{\Mysubsubsection}{\mbox{}\newline %
\noindent\refstepcounter{MyCount}\textbf{\theMyCount.}\ }

\usepackage{xcolor}

\theoremstyle{plain}
\newtheorem{theo}{Theorem}[section]
\newtheorem{cor}[theo]{Corollary}
\newtheorem{prop}[theo]{Proposition}
\newtheorem{lem}[theo]{Lemma}
\newtheorem*{theoU}{Theorem}

\theoremstyle{definition}
\newtheorem{defi}[theo]{Definition}
\theoremstyle{remark}
\newtheorem{rem}[theo]{Remark}

	\title[Galois action and stack inertia]{On Galois action on inertia stack of moduli spaces of curves}
	\author[B.~Collas and S.~Maugeais]{Benjamin Collas and Sylvain Maugeais}
	\thanks{Supported by Prof.~Michael S.~Weiss' Humboldt Grant Professorship, DFG programme DE 1442/5-1 Bayreuth, and DFG SPP 1786}

	\hypersetup{  
	pdfinfo={  
		Title={On Galois action on inertia stack of moduli spaces of curves},  
		Author={Benjamin Collas, Sylvain Maugeais},
		Subject={We establish that the geometric action of the absolute Galois group on the étale fundamental group of moduli spaces of curves induces a Galois action on its  stack inertia subgroups, and that this action is given by \emph{cyclotomy conjugacy}. This result extends the special case of inertia without étale factorisation previously established by the authors. It is here obtained in the general case by comparing deformations of Galois actions. 
		Since the stack inertia corresponds to the first level of the stack stratification of the space, this results, by analogy with the arithmetic of the Deligne-Mumford stratification, opens the way to a systematic Galois study of the stack inertia through the corresponding stratification of the moduli stack.},  
		Keywords={Algebraic fundamental group, inertia stack, special loci, deformation of curves, limits of Galois representation},  
	}  
}
\begin{document}	
	
	\begin{abstract}
		We establish that the geometric action of the absolute Galois group $\mathrm{Gal}(\bar\QQ/\QQ)$ on the étale fundamental group of moduli spaces of curves induces a Galois action on its  stack inertia subgroups, and that this action is given by \emph{cyclotomy conjugacy}. This result extends the special case of inertia without étale factorisation previously established by the authors. It is here obtained in the general case by comparing deformations of Galois actions. 
		
		Since the \emph{cyclic} stack inertia corresponds to the first level of the stack stratification of the space, this results, by analogy with the arithmetic of the Deligne-Mumford stratification, opens the way to a systematic Galois study of the stack inertia through the corresponding stratification of the moduli stack.
	\end{abstract}

	\maketitle
\noindent{\footnotesize\emph{Mathematics Subject Classification (2010).}
	11R32, 
	14H10, 
	14H30, 
	14H45.
}

\noindent{\footnotesize\emph{Keywords.} Algebraic fundamental group, inertia stack, special loci, deformation of curves, limits of Galois representation.}		
	
\section{Introduction}\label{sec:intro}
Let $\M_{g,[m]}$ denote the moduli space of curves of genus $g$ with $m$ unordered marked points endowed with its Deligne-Mumford stack structure over $\QQ$. For a $\QQ$-point of $\M_{g,[m]}$, the choice of a geometric point $\bar x\colon \Spec (\bar \QQ) \to\M_{g,[m]}\otimes\bar{\QQ}$ defines a geometric Galois representation 
\begin{equation}\label{eq:GalGeoRep}
	\rho_{\bar x}\colon\Gq\to \Aut[\pi_1^{et}(\M_{g,[m]}\otimes\bar{\QQ},\bar x)]
\end{equation} 
whose description has been extensively studied following A.~Grothendieck's seminal program \cite{GRO97} in terms of \emph{the Deligne-Mumford stratification} \cite{KNUII} of its stable compactification $\bar{\M}_{g,[m]}$, see \cite{DEL69}. This approach deals essentially with the schematic structure of $\M_{g,[m]}$ by considering the Galois action on the divisorial inertia groups of the boundary of $\bar{\M}_{g,[m]}$, see \cite{NAKA99,NAKA97} ; the arithmetic of the stratification resulted in the development of Grothendieck-Teichmüller theory as initiated by V.~Drinfel'd and Y.~Ihara -- see for example \cite{DRI90,Nakamura2000}.

Let $I_{\mathcal M}$ be the inertia stack of $\M_{g,[m]}$ that classifies the automorphism of curves, i.e. whose objects over a $\QQ$-scheme $S$ consist of pairs $(x, \gamma)$ with $x\in \M_{g,[m]}(S)$ and $\gamma\in\textrm{Aut}_S(x)$. The fibre Over an algebraic closed point $\bar x\colon \Spec (\bar{\QQ})\to \M_{g,[m]}$ gives a finite group $I_{\bar x}=I_{\mathcal M}\times_{\M_{g,[m]}} \Spec(\bar{\QQ})$ called the \emph{stack inertia group of $\bar x$}. This group is isomorphic to the automorphisms group of the curve $\bar x$. The fact that $I_{\bar x}$ injects into $ \pi_1^{et}(\M_{g,[m]}\otimes\bar{\QQ},\bar x)$, proven in \cite{NOOHI04}, raises the questions of the definition and the description of the global geometric $\Gq$-action of Eq.~\eqref{eq:GalGeoRep} on the local stack inertia groups $I_{\bar x}$ of $\M_{g, [m]}$.

The $\Gq$-action on cyclic stack inertia groups gained some focus initially in genus $0$ via Grothendieck-Teichmüller theory \cite{LS97}, then in higher genus with Galois considerations \cite{NAKTSU03}, \cite{Colg1} \S 3.

\bigskip

The main result of the present article follows \cite{Colg0,Colg1,ColMau} and provides an answer to these questions in the case of cyclic stack inertias $I_{\bar x}$, see Theorem \ref{theo:GalActStrg}:
\begin{theoU}[A]
	For any cyclic stack inertia group $I=\mathopen<\gamma\mathclose>$ of $\M_{g,[m]}$, there exists a geometric Galois representation $\rho_{\vec s}$ which induces a $\Gq$-action on $I$ given by $\chi$-conjugacy, i.e. for $\sigma\in\Gq$:
	\[
	\rho_{\vec s}(\sigma).\gamma=\delta_{\sigma}\,\gamma^{\chi(\sigma)}\,\delta_{\sigma}^{-1}\quad \text{for } \delta_{\sigma}\in\pi_1^{et}(\M_{g,[m]}\otimes\bar{\QQ}), 
	\]
	where $\chi\colon\Gq\to\widehat{\ZZ}^{*}$ denotes the cyclotomic character.

\end{theoU}

The present approach follows that of \cite{ColMau} using irreducible components of special loci of the form $\M_{g, [m]}(G)$ -- locus of points $x\in\M_{g,[m]}$ whose geometric stack inertia group $I_{\bar x}$ contains a subgroup isomorphic to $G$ -- which defines a \emph{stack inertia stratification of $\M_{g, [m]}$}, see \cite{DOUAI06}. The key ingredients are the Deligne-Mumford compactification and the arithmetic notion of tangential base point: first to define a $\Gq$-action $\rho_{\vec s}^I\colon \Gq\to \Aut(I)$ that is compatible to a tangential version $\rho_{\vec s}$ of Eq.~\eqref{eq:GalGeoRep}, then to extend the cyclotomy result of ibid.~from the case of stack inertia without étale factorisation (see \S \ref{subsub:totRam} for definition) to the general case.

\bigskip

The definition of an intrinsic \emph{local Galois action on $I_{\bar x}$ within $\pi_1^{et}(\M_{g, [m]}\otimes \bar{\QQ})$} for general Deligne-Mumford stacks is indeed tedious: denoting $\mathcal Z$ the connected component of $\M_{g, [m]}$ containing $\bar x$, and $K$ its field of definition, one obtains an outer action $\Gal(\bar K/K)\to \mathrm{Out}(I)$ \emph{modulo a certain geometric monodromy group only} -- see \cite{LV18} Proposition 2.17. In the case of the cyclic inertia of $\M_{g, [m]}$, we bypass this difficulty by the use of $\QQ$-tangential base points (denoted by $\vec s$) -- obtained from formal neighbourhoods of $\QQ$-points of $\bar{\M}_{g,[m]}$ \cite{IHNAK97}, see \S \ref{subsub:qTBP} for a general stack definition --  and by explicit properties of deformation of $G$-curves \cite{EKE95}, see \S \ref{subsub:degGcurve}. Following a stack version of Grothendieck-Murre formalism of the tame fundamental group \cite{GRO71} -- see in \S \ref{sub:Tbpt}, this leads to some $\Gq$-tangential representations $\rho_{\vec s}$ that induce proper \emph{stack inertia $\Gq$-actions $\rho_{\vec s,\bar x}^I\colon\Gq\to \Aut[I]$}, see \S \ref{subsub:defGalAcIn}.

The property that irreducible components of cyclic special loci are Deligne-Mumford stacks defined over $\QQ$ (see \cite{ColMau}, Proposition 3.12 and Theorem 4.3) ensures their global $\Gq$-invariance by $\rho_{\vec s}$, hence the $\Gq$-stability of the conjugacy class of $\gamma$. While this property is sufficient to establish the cyclotomy result for the first non-trivial cases -- when $G$ is of prime order, or the $G$-action is without étale factorisation (see Proposition \ref{prop:GalActTotRam} then Corollary \ref{theo:GalActTotRam})--, it is not for the general $\M_{g, [m]}(\ZZ/m\ZZ)$. The extension of the cyclotomy result relies, first on the construction of a specific $G$-deformation of smooth curve to the boundary of $\bar{\M}_{g,[m]}$ (Theorem \ref{theo:GoodGAct}), then on the existence of $\Gq$-tangential compatible Knudsen morphisms between the stratas $\M_{g-1,[m]+2}$ and $\bar{\M}_{g,[m]}$ (Proposition \ref{cor:ClutGqComp}), and finally on a specialization result for stack inertia groups, see \S \ref{subsub:specI}. This process takes the name of \emph{inertial limit Galois action} in \S \ref{subsubsec:opening}.

\bigskip

The result of Theorem (A), as well as the use of the inertial limit Galois action in the study of the stack inertia stratification of $\M_{g,[m]}$, strengthens the analogy between the arithmetic of the Deligne-Mumford stratification and of the stack inertia stratifications, which suggests further developments along this direction, see \S \ref{subsubsec:opening}. Theorem (A) also supports a positive answer to the anabelian Question 8.5 of \cite{LOC11}:

\medskip

{\itshape
	If a $\Gq$-action on $\pi_1(\M_{g,[m]}\otimes\bar{\QQ})$ is given by $\chi$-conjugacy on a protorsion element, is this element conjugate to a finite stack inertia one?
}

\medskip

We refer to ibid.~for further motivations and for the original formulation in terms of Dehn twists in the mapping class group $\widehat{\Gamma}_{g,[m]}\simeq \pi_1^{et}(\M_{g,[m]}\otimes\QQ)$, as well as to \cite{NAK90} Theorem 3.4 for the divisorial analog for curves that motivates this question.

\section{Geometry at Infinity of Special Loci}
Let $G$ be a finite group and $\mathcal M_{g,[m]}(G)$ the special loci associated to $G$, i.e. the Deligne-Mumford substack of $\mathcal M_{g,[m]}$ classifying families of smooth and proper marked curves of genus $g$ whose automorphisms group admits a subgroup isomorphic to $G$, see \cite{ColMau} \S2. In case $G$ is cyclic, a certain type of degeneracy of curves in $\mathcal M_{g,[m]}(G)$ to stable curves is sought in the boundary of $\mathcal M_{g,[m]}$: one of the main result is that any smooth $G$-curve admits a degeneracy to an irreducible singular curve with $G$-action whose normalisation is \emph{without étale factorisation} (i.e. the associated $G$-cover does not factorise through a non-trivial étale cover.)

\subsection{Deformation of stable $G$-Curves}\label{sec:Deformation}

For curves endowed with a $G$-action, the analogue of the stable curves \cite{DEL69} is given by \emph{stable $G$-curves} \cite{EKE95}, i.e. stable curves endowed with an \emph{admissible action}, whose definition is recalled below in case of a \emph{cyclic group} -- see also \cite{BERO07} \S 4.1.1.

\begin{defi}[Admissible action]
	Let $G$ be a cyclic group acting faithfully on a semi-stable curve $\mathcal C / S$. For $S$ the spectrum of an algebraically closed field, the action of $G$ is said to be admissible if, for every singular point $P \in \mathcal C$ with stabilizer $G_P$, the two characters of $G_P$ on the branches at $P$ are each other's inverse. For $S$ general, the action is admissible if it is so on every geometric fibre.
\end{defi}

Denoting by $\Omega^1_{C/k}$ the sheaf of relative Kähler differentials, the cohomological theory that controls the $G$-equivariant deformation functor $\textrm{Def}_{C,G}$ is given by the $G$-equivariant $\Ext^{i}_G(\Omega_{C/k}, \O_C)$, see \cite{EKE95} Proposition 2.1. 
It then follows from Schlessinger's theory \cite{SCH06} that the equivariant deformation functor $\textrm{Def}_{C,G}$ associated to a stable $G$-curve is pro-representable by a complete local ring $R_{C,G}$.

\begin{theo}[\cite{EKE95} -- Prop. 2.1--2.2]\label{theo:Eke}
	Let $C$ be a stable $G$-curve over a field $k$ of characteristic $0$ endowed with a $G$-admissible action, and let $R_{C,G}$ be its universal deformation ring with residue field $k$ and field of fraction $K$. Then $R_{C,G}$ is formally smooth over $k$, and its generic point corresponds to a smooth curve over $K$.
\end{theo}

Following \cite{TUF93}, there is no obstruction to infinitesimal lifting; one also obtains a \emph{local-global principle} for tame $G$-covering:
\[
	0\to \H^1_G(C, \Theta_C) \to \Ext^1_G(\Omega^1_{C/k}, \O_C) \to \bigoplus_{I} \Ext^1_G(\widehat{\Omega}_{C/k, P_i}, \widehat{\O}_{C, P_i}) \to 0.
\]
where $\Theta$ is the tangent sheaf and the direct sum is over a set of representatives $\{P_i\}_I$ of $sing(C)/G$. The ``local'' contributions are given by the deformations of the $P_i$, each of them being of dimension $1$.

The universal deformation ring thus identifies to
\[
R_{C,G} \simeq R_{glo} \widehat \otimes k\llbracket q_1, \ldots, q_M\rrbracket,\, \text{where}
\]
\begin{enumerate}
	\item $M$ is the number of singular and ramification points of $C$,
	\item $R_{glo}$ is a formally smooth $k$-algebra of finite dimension,
\end{enumerate} 
by \cite{BERO07} Equation (40). Denoting by $g'$ the genus of $C/G$, one recovers that $\dim (R_{C,G})=3g'-3+b$ with $b$ the degree of the branch points divisor.

\bigskip

In the $G$-stable compactification $\bar{\M}_{g,[m]}(G)$ of $\M_{g,[m]}(G)$, the choice of some deformation parameters $\mathbf{q}=\{q_1,\dots,q_{3g'-3+b}\}$ of $R_{C,G}$ provides a formal neighbourhood $\Spec k\llbracket \mathbf{q}\rrbracket \to \bar{\M}_{g,[m]}(G)$.

\subsection{The Case of $G$-curves without étale factorisation}\label{sub:GDefGenType}
In the following, denote by $G$ a cyclic group of order $n\in \NN$ and fix a generator $\gamma$ of $G$. A a degeneration result is established for $G$-covers in terms of their associated branch data $\bf kr$ by building a specific stable marked $G$-curve and controlling the branching data through $G$-equivariant deformation. The process relies on a rigidification of  Hurwitz data, the $\gamma$-type, first applied to the case of unmarked, then to marked curves.

\Mysubsubsection 
Let $C/k$ be a $G$-curve over a field  $k$ containing $n$-th roots of unity. Then $C\to C/G$ is étale locally  given by an equation of the form
\[
y^n=\prod_I(x-\alpha_i)^{k_i}
\] 
and  the order of the stabilizer group of $\alpha_i$ is given by $n/\gcd(n,k_i)$. Denote by ${\bf k}=\{k_1,\dots,k_{\nu}\}\in (\ZZ/n\ZZ)^{\nu}$ these Hurwitz data associated to the $G$-cover $C\to C/G$. The \emph{$\gamma$-type} of a point of $C$ gives a way to recover the Hurwitz data from local informations.

\begin{defi}[$\gamma$-type]\label{defi:GType}
	Let $k$ be an algebraically closed field of characteristic $0$, let $G$ be a cyclic group of order $n$, $\gamma\in G$ a generator and $\zeta \in k$ a fixed primitive $n$-th root of unity. Let $C/k$ be a complete smooth curve endowed with a $G$-action and $P \in C$ be a closed point with non-trivial stabiliser under the action of $G$. 
	
	The point $P$ is said to be of $\gamma$-type $\zeta$ if, for a uniformising parameter $u$ of $C$ at $P$ we have
	\begin{equation*}
	\gamma^\ell(u)=\zeta^\ell\, u\ \mod u^2 \quad \textrm{for all $\ell \in \ZZ$ such that } \gamma^\ell(P)=P.
	\end{equation*}
	We denote by $\mathrm{type}_{\gamma}(P)$ the $\gamma$-type of a point $P\in C$.
\end{defi}

The $\gamma$-type of a point is independent of the choice of the uniformising parameter $u$. The following lemma gives a link between the local $\gamma$-type and the Hurwitz data of the cover $C\to C/G$, and is used in the next section to build a stable $G$-curve by gluing two points of \emph{inverse $\gamma$-types}.

\begin{lem}\label{lem:unifZeta}
	Let $C/k$ be a complete smooth curve endowed with a $G$-action, and denote by $\{P_i\}_I$ the ramification points of $C\to C/G$ with Hurwitz data $\mathbf{k}=\{k_i\}_I$. Then there exists $\zeta\in k$ such that for all $i\in I$
	\begin{equation*}
	\mathrm{type}_{\gamma}(P_i)=\zeta^{j_i\,\frac{n}{\ord(k_i)}}
	\end{equation*}
	where $j_i$ is the inverse of $k_i\,\frac{\ord(k_i)}{n}$ modulo $\ord(k_i)$.
\end{lem}

Note that for $a \in \ZZ/n\ZZ$, the element $a\frac{\ord(a)}{n}$ is well defined in $\ZZ/\ord(a) \ZZ$.

\begin{proof}
	By Kummer theory, the morphism $\pi\colon C \to C/G$ is given over the étale locus by an equation of the form $y^n=f(x)$ and the action of $G$ is given by $\gamma(y)=\zeta y$. Let $w$ be a uniformising parameter at $\pi(P_i)$ in $C/G$ so that up to a $n$-th power $y^n=w^{k_i}t$ -- where $t$ is an invertible element. 
	
	Writing a decomposition $an+k_i j_i = \frac{n}{\ord k_i}$, the element 
	\begin{equation*}
	u=y^{j_i\frac{n}{\ord k_i}} w^a
	\end{equation*} is a uniformising parameter of $C$ at $P_i$ and we have
	\begin{align*}
	\gamma(u)&=\gamma(y)^{j_i\frac{n}{\ord k_i}} w^a\\
	& =\zeta^{j_i\frac{n}{\ord k_i}}\, y^{j_i\frac{n}{\ord k_i}} w^a\\
	\gamma(u)&=\zeta^{j_i\frac{n}{\ord k_i}} u,
	\end{align*}
	hence the corresponding $\gamma$-type of $P_i$.
\end{proof}

\begin{rem}\label{rem:ActionTangentSpace}
	Once a generator $\gamma$ of $G$ and a primitive $n$-th root $\zeta$ are fixed, the Hurwitz data $\mathbf{k}$ are read locally through the action of $G$ on the tangent space of the ramification points. It is then possible to compute the genus of $C/G$ using Riemann-Hurwitz formula. 
\end{rem}

\Mysubsubsection \label{subsub:degGcurve}
The degeneration of $G$-equivariant curves to the boundary of $\mathcal M_{g}$ is studied, first in the case of unmarked curves.

\begin{theo}\label{theo:GoodGAct}
	Let $g, g'\ge 1$ integers, $G=\ZZ/n\ZZ$, ${\bf k}=(k_1,\dots,k_{\nu})\in (\ZZ/n\ZZ)^\nu$ satisfying \begin{align}
	&2g-2=n(2g'-2)+\sum_i (\ord(k_i)-1)\frac{n}{\ord(k_i)}\label{eq:GoodGAct1}\\
	&\sum_i k_i=0.\label{eq:GoodGAct2}
	\end{align}
	For all $\ell \in (\ZZ/n\ZZ)^*$ there exists a singular curve $C_{\ell}/k$ of genus $g$ endowed with a $G$-admissible action, with $C_{\ell}/G$ of genus $g'$, and such that:
	\begin{enumerate}
		\item the normalisation of $C_{\ell}$ is of genus $g-1$ and its quotient by $G$ has Hurwitz data $(k_1,\dots,k_{\nu}, \ell, -\ell)\in (\ZZ/n\ZZ)^{\nu+2}$;\label{it:GoodGAct1}
		\item the generic $G$-equivariant deformation of $C_{\ell}$ is smooth and has $\bf k$ for Hurwitz data.\label{it:GoodGAct2}
	\end{enumerate}
\end{theo}

The construction below also illustrates the control of the Hurwitz data along $G$-equivariant deformations.

\begin{proof}
	Let $E_0/k$ be a smooth curve of genus $g'-1$ over an algebraically closed field $k$, and let $({\bf k}, \ell, -\ell)\in (\ZZ/n\ZZ)^{\nu+2}$. By Proposition 3.7 \cite{ColMau} there exists a $G$-equivariant cover $E_1 \to E_0$ with Hurwitz data $({\bf k},\ell, -\ell)$. Moreover $E_1$ is of genus $g-1$ since
	\begin{align*}
	n(2(g'-1)-2)+\sum_i (\ord(k_i)-1)\frac{n}{\ord(k_i)}+2(n-1)&=2g-4\\
	&=2(g-1)-2.
	\end{align*}
	Let $\{P_1, \ldots, P_\nu\}$ denote the ramification points with Hurwitz data $\{k_1, \ldots,  k_\nu\}$ and $\{P'_1, P'_2\}$ the points with data $\{\ell,-\ell\}$.
	
	Following lemma \ref{lem:unifZeta}, there exists $\zeta$ such that the $\gamma$-type of $P'_i$ ($i=1,2$) verify:
	\begin{equation}
	\label{eq:comptype}
	\text{type}_{\gamma}(P'_1).\text{type}_{\gamma}(P'_2)=1.
	\end{equation}
	Let $C_{\ell}$ be the curve obtained from $E_1$ by gluing $P'_1$ and $P'_2$ as a point $P'$ as in Fig. \ref{fig:CurveGen1} below. As $\ell$ is prime to $n$, the points $P'_1$ and $P'_2$ are both fixed under $G$ so that the curve $C_{\ell}$ is endowed with a $G$-action. Moreover, this action is admissible thanks to \eqref{eq:comptype}, and it satisfies Property \eqref{it:GoodGAct1} of the theorem since $E_1$ is the normalisation of $C_{\ell}$.
	
	\begin{figure}[H]
		\centering
		\def\svgwidth{0.75\textwidth}
		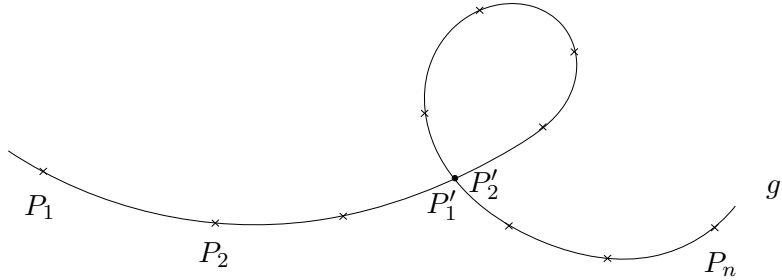
		\caption{Curve $C_\ell$ obtained by gluing $P'_1$ and $P'_2$.}\label{fig:CurveGen1}
	\end{figure}

	By Theorem \ref{theo:Eke}, there exists a $G$-curve $\mathcal C$ over a complete local ring $R$ of residue field $k$, with special fibre $C_{\ell}$ and a generic fibre that is smooth of genus $g$. Let $\gamma$ be a generator of $G$ and $\zeta\in k$ a primitive $n$-th root of unity as fixed by Lemma \ref{lem:unifZeta}. For $P_i\in C_{\ell}$ a ramification point, the action of $\gamma$ on the local ring $\mathcal O_{C_{\ell},Pi}\simeq k[[u_i]]$ is given by $\gamma(u_i)=\zeta^{j_i}u_i$ up to a change of parameter $u_i$. Moreover, there exists a lifting $\tilde P_i$ of $P_i$ such that the action of $G$ on $\mathcal C$ is given in a formal neighbourhood of $\tilde P_i$ by the same formula on the ring $R[[u_i]]$, up to another change of parameter $u_i$. Following Remark \ref{rem:ActionTangentSpace}, the Hurwitz data of the generic fibre of $\mathcal C/R$ is then $\bf k$ as it can be read on $u_i$ through the $\gamma$-type. Finally, a vanishing cycle computation proves that no ramification point of the generic fibre of $\mathcal C/R$ specialises to $P'$, so that $C_{\ell}$ satisfies Property \eqref{it:GoodGAct2}.
\end{proof}

The theorem above is still valid with the assumption $\bf k=\emptyset$, so that a smooth curve with $G$-action of étale type can be built, which specialises to $C_{\ell}$ with only one singular point and whose normalisation has no étale factorisation.

\Mysubsubsection\label{subsub:GDefGenTypeMarked} 
In the case of curves \emph{with $m$ marked points}, i.e. endowed with a horizontal $G$-equivariant Cartier divisor $D$ of degree $m$, the Hurwitz data $\bf k$ are replaced by the branch data $\bf kr$-- see Definition 3.9 in \cite{ColMau}: the \emph{branch data of a curve} $C\in \mathcal{M}_{g,[m]}(G)$  is a couple ${\bf kr}=({\bf k} ,{\bf r})$ where $\bf k$ is a Hurwitz data and ${\bf r}=(r_1,\dots,r_n)$ is a $n$-uple given by:
\begin{equation*}
r_i=\#\{y\in D/G,\text{the branching data at } y \text{ is equal to } i \mod\ n\}
\end{equation*}
where a $n$-th root of unity is fixed. 

\medskip

Let $m'$ denote the degree of the divisor $D/G$. In addition to Equations \eqref{eq:GoodGAct1} and \eqref{eq:GoodGAct2}, we assume that the branch data  also satisfy
\begin{align}
&m=\sum_{i}r_i\gcd(i,n)\\
&m'=\sum_i r_i.
\end{align}

We now state our $G$-deformation result in its complete form.

\begin{cor}\label{Cor:GoodActPoints}
	For any generic point $\eta \in \M_{g, [m]}(G)$, whose corresponding curve satisfies the branch data relations above, there exists a specialisation $z \in \bar\M_{g, [m]}(G)$ of $\eta$ such that the normalisation of the curve corresponding to $z$ has genus $g-1$ and is without étale ramification.
\end{cor}

\begin{proof}
	This is a direct consequence of the description of the set of irreducible components of $\M_{g, [m]}(G)$ by the set of the branching data $\bf kr$ as given in \cite{ColMau}, and of Theorem \ref{theo:GoodGAct}: for a given $\bf kr$, one constructs explicitly a $G$-equivariant marked curve as in the proof above.
\end{proof}

For an algebraic definition of $\bf kr$ (resp. $\bf k$) for families of curves in terms of étale cohomology and for examples see \cite{ColMau} \S 3.2 (resp. \S 3.1). This reference also contains a discussion about the non-canonicity of $\bf k$ and $\bf kr$ relatively to the choice of a generator $\gamma$ of $G$ and a primitive $n$-th root of unity $\zeta\in\mu_n$.

\bigskip

This $G$-deformation result is completed in \S \ref{subsub:finalProof} at the level of automorphism groups of curves. This is then a key ingredient to reduce the $\Gq$-action by cyclotomy to the case of stack inertia without étale factorisation.

\section{Galois Action at Infinity}\label{sec:GaloisInfinity}

Generalising \cite{ZOOAR11} and \cite{NAKA99}, we give the definition of Galois actions at infinity attached to a normal crossing divisor of a generic Deligne-Mumford algebraic stack, and then discuss their compatibility through Knudsen morphisms. We clarify this result in the case of $\mathcal M_{g,[m]}$, leading to non-canonical comparisons of $\Gq$-representations of the fundamental groups of $\mathcal{M}_{g-1,[m]+2}$ and $\mathcal{M}_{g,[m]}$. 

\subsection{Tangential Base Points}\label{sub:Tbpt}
Adapting \cite{ZOOAR11} to the case of a non geometric base point, we define a tangential base point for a Deligne-Mumford algebraic stack $\mathfrak M$ in terms of \emph{the tame fundamental group} of $\mathfrak M$ along a normal crossing divisor, which is adapted from the case of scheme in \cite{GRO71}: for a normal crossing divisor $\mathcal D\to \mathfrak{M}$,  the category $\mathfrak{Rev}^{\mathcal D}(\mathfrak{M})$ of tamely ramified covers of $\mathfrak{M}$ along $\mathcal D$ is defined via the scheme category $ \mathfrak{Rev}^{\mathcal D}(X)$ by pull-back along a presentation $X\to \mathfrak M$ -- see \cite{ZOOAR11} \S3. While being based on Zoonekynd's approach for \emph{geometric} tangential points, the present approach is here developed for rational base points and as such implies the definition of a Galois action.

\Mysubsubsection \label{subsub:qTBP}
Let $k$ be a field and $\mathfrak M$ a Deligne-Mumford algebraic $k$-stack. A point $x \in \mathfrak M(\Spec k)$ is said to have a \emph{Nisnevich neighbourhood} if there exists an étale morphism $f\colon V \to \mathfrak M$ with $V$ a scheme, and a point $v \in f^{-1}(x)$ with residue field $k$. For a point $x$ having a Nisnevich neighbourhood, we define \emph{the local ring at $x$ in $\mathfrak M$}, denoted by $\O_{\mathfrak M, x}^h$, as 
\[
\O_{\mathfrak M, x}^h=\limind_{(v, V)} \O_{V, v}
\]
where the limit is taken over the couples $(v, V)$ as above, and $\O_{\mathfrak M, x}^h$ is the Henselization $\O_{V,v}^h$ for any Nisnevich neighbourhood.
Note that for a stack $\mathfrak M$ over $k$ and a smooth point $x$, the completion of $\O_{\mathfrak M, x}^h$ can be identified to $k[[t_1, \ldots t_n]]$ via the choise of a system of paremeters $t_1, \ldots, t_n \in \O_{\mathfrak M, x}^h$.

Let $\mathfrak M$ be a $k$-stack, $x\in\mathfrak M(\Spec k)$ a smooth point which is supposed to have a Nisnevich neighbourhood and ${\bf t}=\{t_1\ldots, t_n\}$ be a system of parameters of $\O_{\mathfrak M, x}^h$.
We define the \emph{Puiseux ring} of $\mathfrak M$ at $x$ with respect to $\bf t$ as the ring 
\[\tilde\O_{\mathfrak M, x}^{\bf t}=\limind_\ell \left(\widehat{\O_{\mathfrak M, x}^h}\widehat\otimes_k k^{sep}\right)[t_1^{1/\ell}, \ldots, t_n^{1/\ell}].\]

\begin{rem}
	In case of a \emph{geometric point $x$}, the local ring $\O_{\mathfrak M, x}^h$ in the construction of the Puiseux ring above, is replaced by the \emph{strict Henselization} $\O^{sh}_{\mathfrak M, x}$ which exists without condition -- see \cite{LMB00} Rem. 6.2.1.
\end{rem}

The following gives a class of points with Nisnevich neighbourhood, which includes the schematic points of $\mathfrak M$.
\begin{prop}\label{ex:Nisne}
	Let $\mathfrak M$ be a Deligne-Mumford $k$-algebraic stack and $x\in\mathfrak M(\Spec k)$. If $\Aut_k(x)$ is a constant group-scheme, then $x$ admits a Nisnevich neighbourhood.
\end{prop}

\begin{proof} Consider the functor $F\colon (k-\text{Art}) \to \text{Ens}$ defined on the Artinian $k$-algebras by the isomorphism classes
	of objects of $\mathfrak M$ which are deformations of points whose images are equal to that of $x$ in $\mathfrak M$. As
	$\Aut_k(x)$ is a constant group-scheme and the diagonal of $\mathfrak M$ is unramified, the functor $F$ is actually a sheaf for the
	étale topology on $(k-\text{Art})$. Theorem 10.10 of \cite{LMB00} then gives an étale presentation with a $k$-point above $x$, thus a Nisnevich neighbourhood of $x$.
\end{proof}	

\bigskip

Consider a normal crossing divisor $\mathcal D$ on $\mathfrak M$ whose support contains $x$, and let ${ t}_D=\{t_1, \ldots, t_n\}$ be a system of parameters 
of $\mathfrak M$ at $x$ such that $\mathcal D$ is given in an étale neighbourhood of $x$ by $t_1\cdots t_m=0$. A \emph{$k$-rational tangential base point} on $\mathfrak M\setminus \mathcal D$ at $x$ is then defined as a fibre functor in term of \emph{Puiseux ring}: 
\begin{defi}
	\label{defi:tbp}
	Let $\mathfrak M$ be a Deligne-Mumford $k$-stack, $x$ be a smooth $k$-point of $\mathfrak M$ having a Nisnevich neighbourhood, and $\D$ be a normal crossing divisor on $\mathfrak M$ with
	$t_{\D}$ a system of parameters of $\D$ at $x$. The $k$-rational tangential base point associated to $\vec s=(x, {t}_{\D})$ is defined as the functor
	\[\begin{array}{cccc}
	F_x^{t_{\mathcal D}}\colon\mathfrak{Rev}^{\D}(\mathfrak M) & \to & \mathfrak{Set} \\
	Y & \mapsto & \Hom_{\Frac(\mathfrak M)}( \Frac\ Y, \Frac(\tilde{\O}_{\mathfrak M, x}^{t_{\D}}))
	\end{array}\]
	where $\Frac$ denotes the ring of fraction.
\end{defi}
Unlike the classical Grothendieck-Murre theory, the base point is here supposed to belong to the normal crossing divisor.

\bigskip

The following result is essentially Theorem 3.7 of \cite{ZOOAR11}.
\begin{theo}\label{theo:RatTBPisFibFunc}
	Let $\mathfrak M$ be a Deligne-Mumford $k$-stack, $x$ be a smooth $k$-point of $\mathfrak M$ having a Nisnevich neighbourhood, $\D$ be a normal crossing divisor on $\mathfrak M$ and $t_{\D}$ a system of parameters of $\D$ at $x$.	Then the \emph{tangential base point} functor $F_{x}^{t_\D}$ is a \emph{fibre functor}. 
\end{theo}

The proof of this theorem goes by showing that this functor is isomorphic to a functor defined by a geometric point $x'\in\mathfrak M$, which by the theory of étale fundamental group is then a fibre functor on the Galois category $\mathfrak{Rev}^{\D}(\mathfrak M)$. Here $x'$ is given by the generic point of $\Frac(\tilde{\O}_{\mathfrak M, x}^{t_\D})$, which is an algebraic closure of $\mathrm{Frac}(\O_{\mathfrak M, x}^h)$ by Puiseux Theorem since $\mathrm{char}(k)=0$.

\bigskip

In particular, this defines the \emph{arithmetic tame fundamental group based at a $k$-rational tangential base point} 
$\pi_1^\D(\mathfrak M;t_\D,x)=\mathrm{Aut}(F_x^{t_\D})$ as the automorphism group of 
the tangential fibre functor. By base change, this arithmetic fundamental group admits a \emph{geometric tame fundamental group} 
$\pi_1^{\D\times \bar k}(\mathfrak M\otimes \bar k;t_{\D\times \bar k},x)$.

\Mysubsubsection 
Consider the absolute Galois group $\textrm{Gal}(\bar k/k)$ of $k$. Through universal properties, the ring $\widehat{\O}_{\mathfrak M, x}^h\widehat{\otimes} \bar k$ inherits an action of $\textrm{Gal}(\bar k/k)$, and so does $\tilde{\O}_{\mathfrak M, x}^{t_{\D}}$ since the system of parameters $t_{\D}$ is defined over $k$. From this fact, it follows that the functors $F_{x}^{t_{\D}}$ and $\pi_1^\D(\mathfrak M;t_{\D},x)$ are both $\textrm{Gal}(\bar k/k)$-equivariant: 

\begin{prop}\label{prop:GalUnivHen} 
	Let $\mathfrak M$ be a $k$-algebraic stack, and $\vec s=(x,{t}_\D)$ be a $k$-rational tangential base point. Then there exists a $\textrm{Gal}(\bar k/k)$-action: 
	\[
	\rho_{\vec s}\colon\textrm{Gal}(\bar k/k)\longrightarrow\Aut[\pi_1^{\D\otimes \bar k}(\mathfrak M\otimes\bar k;t_{\D\times \bar k},x)]
	\]
	called \emph{a tangential $\textrm{Gal}(\bar k/k)$-representation.}
\end{prop}
A tangential base point is functorial through base change over $k$, while it is not for general $k$-stacks morphisms. Indeed, there are some extra data and assumptions required to define \emph{a stack morphism between tangential base points}. Consider a representable morphism $f\colon \mathfrak N \to \mathfrak M$ between Deligne-Mumford $k$-stacks, $\D$ a normal crossing divisor on $\mathfrak M$, and $x\in \D(\Spec k)$ with $y\in f^{-1}(x)$ a $k$-rational point, and suppose that $x$ and $y$ are both smooth and admit a Nisnevich neighbourhood. 

Following Lemma 2.9 of \cite{ZOOAR11} there is an identification
\[
\Hom_{\Frac(\mathfrak M)}( \Frac\ Y, \Frac(\tilde{\O}_{\mathfrak M, x}^{t_{\D}})) = \Hom_{\mathfrak M}(Y,\tilde{\O}_{\mathfrak M, x}^{t_{\D}}),
\]
so that it is possible to define a tangential base point in terms of rings instead of fields. Suppose morevover that there are two systems of parameters $t_{f^* \D}=\{t'_1, \ldots, t'_n\}$ and $t_{\D}=\{t_1, \ldots, t_\ell\}$ respectively of $f^* \D$ at $y$ and of $\D$ at $x$ such that the induced morphism $f^h\colon \O^h_{\mathfrak M, x}\to \O^h_{\mathfrak N, y}$ sends $t'_j$ to $t_j$ or $0$. Then the $\textrm{Gal}(\bar k/k)$-equivariant natural transformation of functors
\[
F_f\colon F_{x}^{t_{\D}} \to F_{y}^{t_{f^*\D}},
\]
is obtained using the $\textrm{Gal}(\bar k/k)$-equivariant morphism
\[
\tilde f\colon \tilde\O^{t_{\D}}_{\mathfrak M, x}\to \tilde\O^{t_{f^* \D}}_{\mathfrak N, y}.
\]

\bigskip

In the two important cases of unramified and smooth morphisms, an extension property guarantees the following nearly ``functorial'' result.

\begin{prop}\label{prop:compMorph}
	Let $f\colon \mathfrak N \to \mathfrak M$ be a morphism of Deligne-Mumford $k$-stacks, $\D$ be a normal crossing divisor on $\mathfrak M$ such that $f^*\D$ is a normal crossing divisor on $\mathfrak N$, and let $x\in \D(\Spec k)$ and $y\in f^{-1}(x)$ be $k$-rational points such that both $x$ and $y$ are smooth and have Nisnevich neighbourhoods. 
	
	If $f$ is either smooth or unramified, then there exist regular systems of parameters $t_{f^* \D}$ and $t_{\D}$ of $f^* \D$ at $y$ and $\D$ at $x$, and a Galois equivariant morphism
	\[
	\pi_1^{f^* \D}(\mathfrak N; t_{f^*\D}, y) \to \pi_1^{\D}(\mathfrak M; t_{\D}, x).
	\]
\end{prop}

\begin{proof}
	Following the discussion above, it is sufficient to establish that there exist two systems of parameters $t_{\D}$ for $\D$ at $x$ and $t_{f^*\D}$ for $f^*\D$ at $y$ such that $f$ induces a morphism $f^h\colon \O^h_{\mathfrak M, x}\to \O^h_{\mathfrak N, y}$ that sends an element of $t_{\D}$ on an element of $t_{f^*\D}$ or $0$.
	
	If $f$ is unramified, the morphism $f^\#\colon \widehat\O^h_{\mathfrak M, x}\to \widehat\O^h_{\mathfrak N, y}$ is a surjection. Consider $t_1, \ldots, t_n$ a system of parameters of $\D$ in $\widehat\O^h_{\mathfrak M, x}$. As $f^*\D \not = \emptyset$, we have $f^h(t_1\cdots t_\ell)\not = 0$ so that it is possible to extract from $f^\#(t_1), \ldots, f^\#(t_n)$  a system of generators of $\widehat \O_{\mathfrak N, y}$ because $f^\#$ is formally unramified and induces an injection on tangent spaces.
	
	\medskip
	
	If $f$ is smooth, then the morphism $f^\#$ is injective and a system of parameters $t_1, \ldots, t_n$ for $\D$ in $\widehat\O^h_{\mathfrak M, x}$ can be completed into a system of parameters $t_1, \ldots, t_{n'}$ of $\widehat\O^h_{\mathfrak N, y}$ by picking up vectors in the tangent space.
\end{proof}

The lack of functoriality comes from the fact that \emph{there is no obvious choice of parameters}, which has the important consequence below.
\begin{rem}\label{Rem:GalNonFunc}
	A $k$-rational change of parameters $t_{\D}$ to $t'_{\D} $ -- or \emph{infinitesimal homotopic transformation} -- leads to two \emph{$k$-homotopically} equivalent $k$-rational base points $F_x^{t_\D} \simeq F_x^{t'_\D}$, but \emph{not to equivalent} $\textrm{Gal}(\bar k/k)$\emph{-actions} on the fundamental groups: the action on Puiseux series makes some Kummer character appear from the $N$-th roots of the involved rational coefficients.
\end{rem}

\Mysubsubsection\label{subsub:defGalAcIn}
Let us fix a Galois representation defined by the choice of a $k$-tangential base point $\vec s$ on $\mathfrak M$ as in Proposition \ref{prop:GalUnivHen}
\[
\rho_{\vec s}\colon\Gal(\bar k/k)\to \Aut[\pi_1^{\D\otimes \bar k}(\mathfrak M\otimes \bar k; \vec{s})]
\] 
. We are giving here an extra assumption to define naturally a $\Gal(\bar k/k)$-action on the stack inertia groups of $\mathfrak M$ using $\rho_{\vec s}$.

Consider a geometric point $w\colon \Spec(\bar k)\to \mathfrak M$  of $\mathfrak M$ and let $I_{\mathfrak M,w}=\Spec(\bar k)\prescript{}{w}{\times} I_{\mathfrak M}$ denote its stack inertia group of $2$-transformations, with $I_{\mathfrak M}=\mathfrak M \times_{\mathfrak M \times \mathfrak M} \mathfrak M$ denoting the inertia stack of $\M$. Since $\gamma \in I_{w}=I_{\mathfrak M, w}$ induces a transformation of the fibre functor $F_{w}$, this defines a morphism $\omega_{w} \colon I_{w} \to \pi_1(\mathfrak M\times\bar k, w)$ and $I_{w}$ is also called the group of \emph{hidden paths} of the étale fundamental group, cf. \cite{NOOHI04} \S 4.

Consider now $\bar z \colon \Spec(\bar K) \to \mathfrak M$ a geometric point, and choose an injection $\bar k \subset \bar K$. As any $\sigma \in \Gal(\bar k/k)$ can be extended to a $k$-automorphism $\tilde \sigma$ of $\bar K$, fix one such $\tilde\sigma$ and define $\tilde\sigma(\bar z)$ by base change. The choice of two étale paths from $\bar z$ to $\vec s$ and from $\tilde{\sigma}(\bar z)$ to $\vec s$, then defines morphisms 
\begin{align*}
\phi \colon \pi_1(\mathfrak M, {\bar z})\longrightarrow\pi_1^{\D}(\mathfrak M; \vec{s})\quad\text{and}\quad
\phi_{\tilde \sigma} \colon  \pi_1(\mathfrak M, \tilde\sigma(\bar z))\longrightarrow  \pi_1^{\D}(\mathfrak M; \vec{s}).
\end{align*}
Moreover, the compatibility between $\sigma$ and $\tilde\sigma$ induces a diagram
\begin{equation}
\label{diag:InertialAction}
\begin{tikzcd}[column sep=small]
I_{\bar z} \arrow{d}{\omega_{\bar z}}\arrow{rrr}{\tau \mapsto \tilde\sigma^{-1}\tau\tilde\sigma} & & & I_{\tilde\sigma(\bar z)}\arrow{d}{\omega_{\tilde{\sigma}(\bar z)}} \\
\pi_1(\mathfrak M\otimes \bar k, {\bar z}) \arrow{rd}{\phi} & & & \pi_1(\mathfrak M\otimes \bar k, \tilde{\sigma}(\bar z))\\ 
& \pi_1^{\D\otimes\bar k}(\mathfrak M\otimes \bar k; \vec{s}) \arrow{r}{\sigma} & \pi_1^{\D\otimes \bar k}(\mathfrak M\otimes \bar k; \vec{s})  \arrow{ru}{\phi_{\tilde \sigma}^{-1}} & 
\end{tikzcd}
\end{equation}
where the bottom line is the action by conjugacy defined by the tangential $\Gal(\bar k/k)$-action $\rho_{\vec s}$ on $\pi_1^{\D\times \bar k}(\mathfrak M\otimes \bar k, \bar z)$. This diagram is commutative \emph{up to conjugacy by a hidden path} $\epsilon$ from $\bar z$ to $\tilde\sigma(\bar z)$, i.e. a $2$-transformation $\xymatrix@C4em{\Spec(\bar K) \rtwocell^{\bar z}_{\tilde\sigma(\bar z)}{\epsilon} &\mathfrak M}$.

\bigskip

In case $\bar z$ is stable under $\Gal(\bar k/k)$ -- for example as in Proposition \ref{rem:action_intertielle} -- the choice of $\phi_{\tilde\sigma} = \phi$ for all $\sigma$ is possible and makes the Diagram \eqref{diag:InertialAction} strictly commutative, i.e. without any hidden path conjugacy by $\epsilon\colon\bar z\Rightarrow\sigma(\bar z)$. Since $\rho_{\vec s}$ sends the image $\omega_{\bar z}(I_{\bar z})$ into itself in the first line of Diagram \eqref{diag:InertialAction}, the tangential $\Gal(\bar k/k)$-action $\rho_{\vec s}$ defines a \emph{$\Gal(\bar k/k)$-stack inertia Galois action $\rho_{\vec s,\bar z}^I$}
\begin{equation}\label{eq:InGalSt}
	\rho_{\vec s,\bar z}^I\colon\Gal(\bar k/k)\longrightarrow \Aut[I_{\bar z}].
\end{equation}
Mover over, the action $\rho_{\vec s,\bar z}^I$ is induced by the canonical local $\Gal(\bar K/K)$-action $\rho_{\bar z}^I$, via the commutating diagramme:
\begin{equation}
	\xymatrix{
		\Gal(\bar K/K)\ar[r]^{\rho_{\bar z}^I}\ar[d]_\pi& \Aut[I_{\bar z}] \ar@{=}[d]\\
		\Gal(\bar k/k)\ar[r]^{\rho_{\vec s,\bar z}^I}& \Aut[I_{\bar z}]
}
\end{equation}
where $\pi\colon\Gal(\bar K/K)\to \Gal(\bar K\otimes \bar{k}/k)\simeq \Gal(\bar k/k)$ is surjective when $K/k$ is linearly disjoint from $\bar k/k$.

\begin{prop}\label{rem:action_intertielle}
	Let $z\colon\Spec(K)\to \mathfrak M$ be a $K$-point, $\bar z\colon \Spec (\bar K) \to \mathfrak M$ a geometric point above $z$ and $I_{\bar z}\to \pi_1(\mathfrak M\otimes \bar k,\bar z)$ its stack inertia, and suppose that $K/k$ is linearly disjoint from $\bar k/k$. Then any Galois representation $\rho_{\vec s}\colon\Gal(\bar k/k)\to \Aut[\pi_1^{\D\otimes \bar k}(\mathfrak M\otimes \bar k; \vec{s})]$ defined by a $k$-tangential base point $\vec s$ on $\mathfrak M$ defines a $\Gal(\bar k/k)$-action $\rho_{\vec s,\bar z}^I$ on $I_{\bar z}$ and this action coincide with the action $\rho_{\bar z}^I$ of $\Gal(\bar K/K)$ on $I_{\bar z}$ by conjugacy.
\end{prop}

\begin{proof} This follows from the discussion above. Since $K/k$ is linearly disjoint from $\bar k/k$, the $\bar k$-image of $\bar z\colon \Spec (\bar K) \to \mathfrak M$ is stable under the $\Gal(\bar k/k)$-action.  The tangential $\Gal(\bar k/k)$-action $\rho_{\vec s}\colon  \sigma \mapsto \phi^{-1}\, \sigma\, \phi$ on $\pi_1^{\D\otimes \bar k}(\mathfrak M\otimes \bar k, \bar z)$ then sends the image $I_{\bar z}$ into itself, and so induces an action $\rho_{\vec s,\bar z}^I$ of $\Gal(\bar k/k)$ on $I_{\bar z}$ according to the commutativity of the Diagram \eqref{diag:InertialAction}. 
	
	This proves furthermore, that $\rho_{\bar z}^I$ and $\rho_{\vec s}^I$, seen as an action of $\Gal(\bar K/K)$ through the surjection $\Gal(\bar K/K) \to \Gal(\bar k/k)$, define the same action on $I_{\bar z}$.
\end{proof}

\begin{rem}\label{rem:GqIn}\mbox{~}
	\begin{enumerate}
		\item In the case where $\M=\M_{g, [m]}$ and $I_{\M,w}=G$ is cyclic, one shows that $K$ above can be taken as the field of moduli-definition of the generic point of an irreducible component of the special loci $\M_{g, [m]}(G)$ -- see \cite{ColMau} Lemme 5.2 and Corollaire 4.2.
		\item \label{rem:CompActionInertia}
		Consider a given tangential morphism $f\colon \mathfrak N \to \mathfrak M$ as in Proposition \ref{prop:compMorph} -- i.e. a point $y\in \mathfrak N(\Spec \bar K)$ with image $x \in \mathfrak M(\Spec \bar K)$ with all the compatible datas. The compatible $\Gal(\bar k/k)$-representations in $\pi_1^{f^* \D}(\mathfrak N; t_{f^*\D}, y)$ and $\pi_1^{\D}(\mathfrak M; t_{\D}, x)$ then induce compatible $\Gal(\bar k/k)$-actions $\rho_{\vec s,\bar x}^I$ and $\rho_{\vec s,\bar y}^I$ on the respective inertia groups $I_x$ and $I_{y}$ by commutativity of Diagram \eqref{diag:InertialAction}.
	\end{enumerate}
\end{rem}

\bigskip

Proposition \ref{rem:action_intertielle} and the compatibility through Knudsen morphism are applied in various situations to the moduli spaces of curves in \S \ref{sec:GalActInert}.

\subsection{Tangential Galois Action and Clutching Morphisms}\label{subsec:TangCluCurve}
The tools built in the previous section are now applied to describe more explicitly the tangential $\Gq$-action in the case of the Deligne-Mumford $\QQ$-stack of moduli space of stable curves $\bar\M_{g,[m]}$ and their link through of Knudsen morphisms. Since the case of the stack inertia groups requires an additional specific result, it is dealt with in \S \ref{subsub:DefGqInAct} and \S\ref{subsub:totRam}.

\Mysubsubsection \label{subsub:GqActTBT}
The first step is the choice of a tangential base point in $\mathcal M_{g,[m]}$. Let $x\in \bar\M_{g,[m]}(\Spec\QQ)$ be a maximally degenerated $\QQ$-curve defined as a graph of $\PP^1$ such that marked points and singular points are rational, so that $x$ has only rational automorphisms. Then by Proposition \ref{ex:Nisne},  $x\in \bar\M_{g,[m]}(\QQ)$ admits a Nisnevich neighbourhood. 

Examples of such curves are given by \cite[Fig. $(ii)_n$, $(iii)'_{k,n}$]{IHNAK97} and are reproduced in Fig. \ref{Fig:maxDegG1G2} below for $g\geqslant 1$. Let us denote them by $X_A$ and $X_B$.
\begin{figure}[h]
	\centering
	\subfloat[Curve $X_A$ - type $g=1,n\geqslant 1$]{
		\includegraphics[width=.5\textwidth]{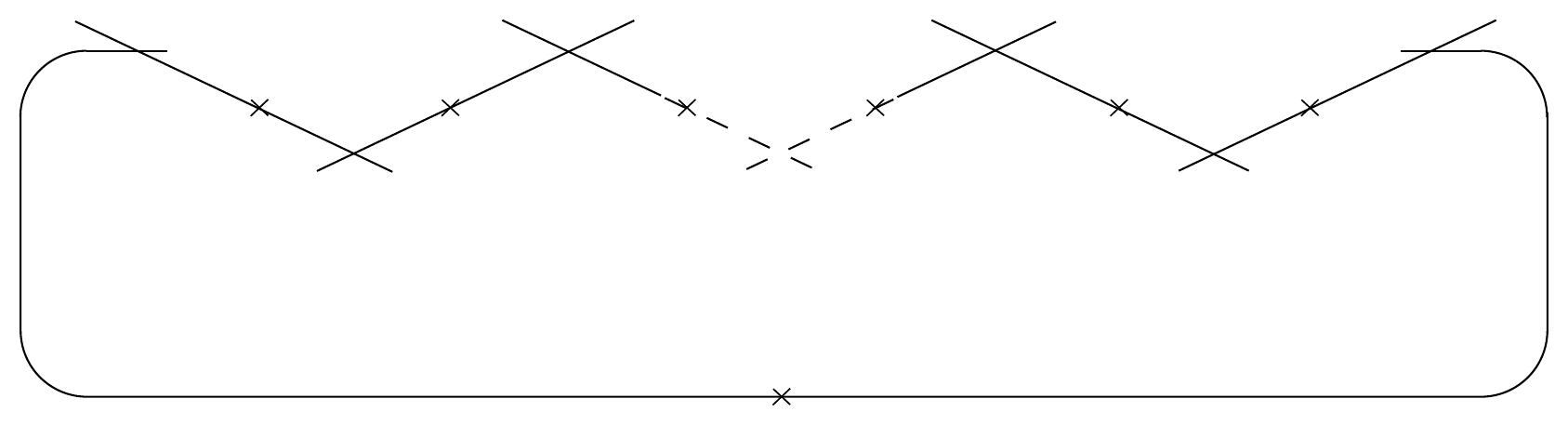}
	}
	\subfloat[Curve $X_B$ - type $g\geqslant 2,n\geqslant 1$]{
		\includegraphics[width=.5\textwidth]{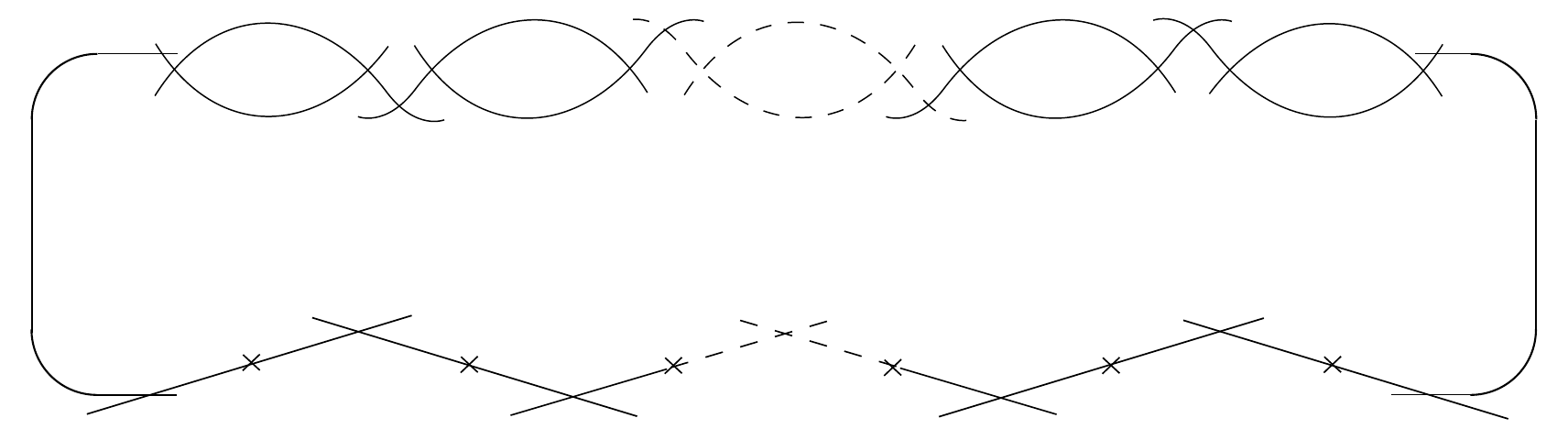}
	}
	\caption{Maximally degenerated curves of type $(g,n)$}
	\label{Fig:maxDegG1G2}
\end{figure}

\begin{rem}\label{rem:tgbptOrig}
  The present construction in \S \ref{subsub:qTBP} is complementary to the original approach of \cite{IHNAK97} by \emph{tangential base point on $\mathcal  M_{g,m}$} where a maximally degenerated curve $X$ by a $\PP^1_{0,1,\infty}$-diagram is defined and a canonical choice of a set of coordinates $\bf q$ of the universal deformation ring $\O^{def}_{X}$ of $X$ is fixed. This corresponds exactly to the choice of a system of parameters which led to Definition \ref{defi:tbp}.
\end{rem}

\Mysubsubsection \label{subsub:DefDivInf}
Let us consider the Knudsen's clutching morphism between moduli spaces of stable curves as defined in \cite{KNUII}:
\[
	\bar\M_{g-1,m+2}\longrightarrow \bar\M_{g,m}.
\]

Considering the action of the permutation group $\mathfrak S_m$ on the first $m$ marked points, the quotient of the clutching morphism by this action defines a morphism
\[
	\beta\colon\bar\M_{g-1,[m]+2}\longrightarrow \bar\M_{g,[m]}.
\]

 Let $\mathcal E= \bar\M_{g, [m]} \setminus \M_{g, [m]}$, and let $\mathcal D$ be the closure of $\mathcal E \setminus \Im(\beta)$ in $\bar\M_{g, [m]}$, which is the union of the irreducible components of $\mathcal E$ disjoint with $\Im(\beta)$. Let $\bar\M_{g-1,[m]+2}$ denote the moduli space of stable curves with $m$ marked points and $2$ fixed points. Then $\mathcal D$ is a normal crossing divisor in $\bar\M_{g, [m]}$ and $\beta^*(\mathcal D)$ is a normal crossing divisor in $\bar\M_{g-1, [m]+2}$ that is equal to $\bar\M_{g-1, [m]+2} \setminus \M_{g-1, [m]+2}$. The partial compactification is then defined as $\tilde\M_{g, [m]}=\bar\M_{g, [m]}\setminus \mathcal D$, and we denote by $\tilde\beta$ the induced Knudsen morphism
\[
	\tilde\beta\colon\bar\M_{g-1, [m]+2}\longrightarrow\tilde\M_{g, [m]}.
\] 

The partial normalisation $X'_A$ (resp. $X'_B$) of $X_A$ (resp. $X_B$) at a singular point $P$, pointed at the pre-images of $P$, is naturally a curve in $\bar\M_{g-1, [m]+2}$ that is sent to $X_A$ (resp. $X_B$) by $\tilde\beta$.


\bigskip

In the following, the fundamental group is based at a point $x\in\{X_A, X_B \}$ in the {partial compactification $\tilde{\M}_{g,[m]}$}. Recall that the choice of a tangential base point $\vec{s}=(x, t_{\mathcal D})$, $x\in \{X_A,X_B\}$, defines the fundamental group $\pi_1^{\mathcal D}(\bar\M_{g, [m]};\vec{s})$, denoted by $\pi_1(\tilde\M_{g, [m]};\vec{s})$ which is isomorphic to $\pi_1(\tilde\M_{g, [m]};x)$ for any geometric point $x \in \tilde\M_{g, [m]}$ by Theorem \ref{theo:RatTBPisFibFunc}. For a $\beta$-compatible tangential base point $\vec s\, '$ we denote in the same way the fundamental group $\pi_1^{\mathcal \beta*(\mathcal D)}(\bar\M_{g-1, [m]+2};\vec{s}\ ')$ by $\pi_1(\M_{g-1, [m]+2};\vec{s}\ ')$.

\Mysubsubsection \label{subsub:GqActKnuMor}
Recall that by Artin-Mazur étale homotopy type theory applied to the $\QQ$-stack $\mathcal M_{g,[m]}$ as in \cite{ODA}, the étale fundamental group associated to a geometric point $\bar x\colon \Spec(\bar{\QQ})\to\mathcal M_{g,[m]}$ yields an Arithmetic-Geometric (short) Exact Sequence:
\begin{equation*}
	1\to \pi_1^{et}(\mathcal M_{g,[m]}\otimes\bar{\QQ},\bar x)\to \pi_1^{et}(\mathcal M_{g,[m]},\bar x)\to \Gq\to 1.
\end{equation*}
For some given normal crossing divisor $\mathcal D$ and $\QQ$-tangential base point $\vec s=(x,t_{\mathcal D})$ on $\bar\M_{g,[m]}$, Proposition \ref{prop:GalUnivHen} 
yields a tangential $\Gq$-representation:
\begin{align}
&\rho_{\vec s}\colon\Gq\to \Aut[\pi_1(\tilde\M_{g,[m]}\otimes \bar \QQ;\vec s)]\label{eq:phi}
\end{align}
where $\pi_1(\tilde\M_{g,[m]}\otimes \bar \QQ;\vec s)$ denotes $\pi_1^{\mathcal D\otimes \bar \QQ}(\bar\M_{g,[m]}\otimes \bar \QQ;\vec s)$.

\bigskip

With the notation of \S \ref{subsub:DefDivInf}, since $\beta$ is \emph{unramified} by \cite{KNUII} Corollary 3.9, we construct some $\beta$-compatible tangential $\Gq$-representations:

\begin{prop}\label{cor:ClutGqComp} There exists a choice of $\QQ$-tangential base points $\vec s$ and $\vec s\,'$ of type $X_A$ or $X_B$, respectively on $\bar{\mathcal M}_{g,[m]}$ and on $\bar{\mathcal M}_{g-1,[m]+2}$ which induces a morphism:
	\[
	\pi_1(\mathcal M_{g-1,[m]+2}\otimes \bar \QQ;\vec s\,')\longrightarrow \pi_1(\tilde\M_{g,[m]}\otimes \bar \QQ;\vec s)
	\]
	and is $\Gq$-equivariant  with respect to $\rho_{\vec s}$ and $\rho_{\vec s'}$.
\end{prop}

\begin{proof}
	Let us first consider the case of ordered marked points $\bar\M_{g,m}$. As the morphism $\bar\M_{g-1,m+2}\to \bar\M_{g,m}$ is unramified by \cite{KNUII} Corollary 3.9, Proposition \ref{prop:compMorph} insures the existence of $\QQ$-tangential base points based at curves of types $X_A$ and $X_B$ and compatible with $\tilde\beta$. Since a tangential representation is defined by the $\Gq$-action on the parameters $t_D$ and on the $\bar{\QQ}$-coefficients of $\vec s$, the result follows directly from ibid. applied to the arithmetic fundamental groups $\pi_1(\mathcal M_{g-1,m+2};\vec s\,')$ and $\pi_1(\mathcal M_{g,m};\vec s)$
	
	For the unordered case, we consider the cartesian diagram
		\[\xymatrix{
			\bar\M_{g-1,m+2}\ar[r]\ar[d]_{\pi_1}&\bar\M_{g,m}\ar[d]^{\pi_2}\\
			\bar\M_{g-1,[m]+2}\ar[r]&\bar\M_{g,[m]}
		}
		\]
		where vertical morphisms are étale surjective since the marked points are supposed distinct. By descent property, to be unramified is local at the source for the étale topology, thus the bottom morphism is unramified and the result follows from the unordered case and the same arguments.
\end{proof}

We insist on the fact that Knudsen morphisms do not lead to \emph{canonical} $\Gq$-actions -- see Remark \ref{Rem:GalNonFunc}. The comparison of $\Gq$-action by change of parameters illustrates the non $\Gq$-invariance of \emph{analytic continuation}, and is indeed \emph{the core of the Arithmetic Geometry of moduli spaces of curves} as illustrated by the role of Deligne's \emph{droit chemin $p$ from $\vec{01}$ to $\vec{10}$} in $\M_{0,4}$ as in \cite{IHARA91}.

\begin{rem} As a special case and as another general application of Proposition \ref{prop:compMorph}, we signal the following:
	\begin{enumerate}
		\item The above construction of $\beta$-compatible  $\Gq$-representations is the algebraic generalisation of the topological approach of \cite{Colg1} where a \emph{mapping class groups} morphism $\Gamma_{0,[m]}^2\to\Gamma_{1,[m]}$ is defined to deal with the étale type inertia in genus $1$;
		\item \label{Rem:KnudGen} The Knudsen's clutching morphisms 
		\begin{align*}
		&\beta_{g_1,g_2}\colon\bar\M_{g_1,m_1}\times \bar\M_{g_2,m_2}\longrightarrow \bar\M_{g,m}.
		\end{align*}
		being closed immersions, the approach above readily applies to the study of various $\beta_{g_1,g_2}$-compatible $\Gq$-representations.
	\end{enumerate}
\end{rem}

\section{Galois Action on Inertia}\label{sec:GalActInert}
This section details the the main result of this paper: the description of the $\Gq$-action on the cyclic stack inertia of $\M_{g,[m]}$ defined by a tangential $\Gq$-representation. First, the approach using
irreducible components of \emph{special loci} initiated in \cite{ColMau} is recalled, and is shown to provide a favourable context for applying \S \ref{sec:GaloisInfinity}. The behaviour of the inertia Galois action under specialisation is then established and used to prove the main theorem as a result of the previous sections.

\subsection{Special Loci and Inertia Groups} 
Let $w$ be a geometric point of $\M_{g,[m]}$ and $G < I_w$ its stack inertia group. We consider the \emph{special loci} $\M_{g,[m]}(G)$ associated to $G$, i.e. the loci of curves of $\M_{g,[m]}$ that admit a $G$-action -- see \cite{ColMau} \S 2.1. By the residual finiteness property of the orbifold fundamental group of $\M_{g, [m]}(\CC)^{an}$, the morphism of \S \ref{subsub:defGalAcIn} turns into an embedding $\omega_{w} \colon I_{w} \hookrightarrow \pi_1(\M_{g,[m]}\otimes\bar \QQ, w)$. 

Let $\rho_s\colon\Gal(\bar \QQ/\QQ)\to \Aut[\pi_1(\tilde\M_{g,[m]}\otimes \bar \QQ; \vec{s})]$ be a Galois tangential representation defined by a $\QQ$-tangential base point $\vec s$ with support in the partial compactification $\tilde\M_{g,[m]}$ as in \S \ref{subsub:GqActKnuMor}. In case $w$ is stable under $\Gal(\bar \QQ/\QQ)$, this defines  a stack inertia $\Gq$-action $\rho_{\vec s}^I$ on $I_w$ (cf. Eq.~\eqref{eq:InGalSt}), that we further describe when $G=\langle\gamma\rangle$ is cyclic.

\Mysubsubsection \label{subsub:DefGqInAct}
The study of the action $\rho_{\vec s}^I$ is linked to the action of $\Gal(\bar \QQ/\QQ)$ on the set of irreducible components of $\M_{g,[m]}(G)\otimes\bar\QQ$ cf. \cite{ColMau}, \S 2.2, so it is fundamental to have first a good description of it. As the normalisation of $\M_{g,[m]}(G)$ is $\M_{g,[m]}[G]/\Aut(G) $, it is possible to replace one by the other, and then prove that the irreducible components of $\M_{g,[m]}(G)$ are geometrically irreducible, see \cite{ColMau} Corollary 3.12 and Theorem 4.3: Denoting $\bf \underline{kr}$ the \emph{algebraic branch data} of \S \ref{subsub:GDefGenTypeMarked}, it is shown that such irreducible component is of the form $\mathcal M_{g,[m],\bf \underline{kr}}(G)$, composed of curves with given $\bf \underline{kr}$ data, see ibid. This result can be reformulated in terms of points of $\mathcal M_{g,[m]}(G)$ with value in precise fields, as in the following lemma.

\begin{lem}[\cite{ColMau} Lemma 5.2]\label{lem:morphlindisjoint}
	For any irreducible component $\mathcal Z \subset \mathcal M_{g,[m]}(G)$ there exists a morphism $\Spec (K) \to \mathcal Z$ with 
	$K$ linearly disjoint from $\QQ$.
\end{lem}

Let $G < I$ be the generic stack inertia group of an irreducible component $\mathcal Z$ of the special loci $\mathcal M_{g,[m]}(G)$. Lemma \ref{lem:morphlindisjoint} provides a $K$-point $z$ of the component whose geometric inertia $I_{\bar z}$ contains $I$, and it follows from Proposition \ref{rem:action_intertielle} that there exists a stack inertia Galois action
\begin{equation}\label{eq:GKI}
		\rho_{\vec s,\bar z}^I\colon\Gq\longrightarrow \Aut[I]
\end{equation}
which is defined by conjugacy and is induced by the local $\Gal(\bar K/K)$-action $\rho_{\bar z}^I$ on $I_{\bar z}$. 

\medskip

The definition of $\rho_{\vec s,\bar z}^I$ relies indeed on many choices, such as fixing an algebraic closure of $K$ or choosing a specialization morphism from $\bar z$ to the boundary of $\bar{\M}_{g,[m]}$. This results in the identification of $\rho^I_{\bar z} $ to the stack inertia $\Gq$-action $\rho^I_{\vec s}$ of Eq. \ref{eq:InGalSt} which is induced by the given tangential $\Gq$-action $\rho_{\vec s}$ -- see the discussion above Proposition \ref{rem:action_intertielle}.

\begin{rem}
 The $K$-point of an irreducible component given by Lemma \ref{lem:morphlindisjoint} is built by factorisation through a certain base change of $\M_{g,[m]}$ in order to kill the automorphisms of the gerbe at the generic point. In particular, when the irreducible component admits a dense open subset with trivial automorphism group, $\Spec (K)\to \mathcal Z$ can be chosen to be the generic point of the component.
\end{rem}

\Mysubsubsection\label{subsub:totRam} For a general curve $C\in \M_{g,[m]}(G)$, recall that the associated $G$-cover $C\to C/G$ factorises as below with the properties:
\vspace*{-1em}

\noindent\begin{minipage}{.78\textwidth}
	\begin{enumerate}
		\item the group $H < G$ is generated by the stabilisers of ramification points of the $G$-cover $C\to C/G$;
		\item the cover $C/H \to C/G$ is étale.
	\end{enumerate}
\end{minipage}
\begin{minipage}{.3\textwidth}
	\[
	\xymatrix@R=1em@C=1.5pt{
		C\ar[dd]\ar[rd]& \\
		& C/H\ar[ld]\\
		C/G& 
	}
	\]
\end{minipage}
When $H=G$, the action of $G$ on $C$ is said to be \emph{without étale factorisation}. In this case, the stack inertia $\Gq$-action of Eq.~\eqref{eq:GKI} is given by the Proposition below, which also plays a key role in the final proof of the general case.

\begin{prop}\label{prop:GalActTotRam}
	Let $\eta\colon \Spec K \to  \mathcal M_{g,[m]}(G)$ be a morphism with value into a field $K$ linearly disjoint from $\QQ$. If the curve $\eta$ is without étale factorisation, then the $\Gq$-action $\rho_{\vec s}^I$ on $G=I_{\bar \eta}$ is given by cyclotomy i.e. for $\sigma \in \Gq$ and $\gamma \in G$ we have $\sigma.\gamma  = \gamma^{\chi(\sigma)}$.
\end{prop}

The proposition is actually \cite{ColMau} Theorem 5.4 to which we refer for details. The idea of the proof is go as follows: Let $C\colon \Spec K \to \mathcal M_{g,[m],{\bf \underline{kr}}}(G)$ be a morphism as in Lemma \ref{lem:morphlindisjoint}, and suppose that the action of $G$ on the curve $C$ is without étale factorisation. Since the stabilisers of a ramification point are generating subgroups of the stack inertia group, the \emph{branch cycle argument} implies that the $\Gq$-action is given by cyclotomy \emph{on a generator $\gamma$ of the inertia group}. 

\bigskip

The following section establishes a similar result for tangential $\Gq$-action $\rho_{\vec s}$ on the fundamental group \emph{for curves with possible étale factorisation} using the $\Gq$-compatibility of the Knudsen morphism of \S \ref{subsub:GqActKnuMor}.

\subsection{Inertial Limit Galois Action and Cyclotomy}\label{subsec:InLimCyc}
We describe the tangential $\Gq$-action $\rho_{\vec s}$ on cyclic stack inertia, first for curves without étale factorisation, then in the general case. Note that the results of the first section readily extend to any Deligne-Mumford stack.

\Mysubsubsection \label{subsub:specI}
Let $\tilde{\mathcal M}_{g,[m]}$ be the partial compactification of \S \ref{subsub:GqActKnuMor}. We establish the behaviour under specialisation of the stack inertia groups within $\pi^{et}_1(\tilde{\mathcal M}_{g,[m]})$. More precisely, the goal of this section is compare different Galois action on the étale fundamental group  based on different points/tangential points.

Let $\vec s$ be a tangential base point of $\tilde{\mathcal M}_{g,[m]}$, $\eta\in \tilde \M_{g,[m]}$ be a point above the generic point of $\vec s$, and $z \in \tilde{\mathcal M}_{g, [m],\bf \underline{kr}}(G)$ be a specialisation of $\eta$. More precisely, let $R$ be a valuation ring with algebraically closed fraction field $K$ and residue field $k$, endowed with a morphism $T\colon \Spec R \to \mathcal M_{g, [m]}(G)$ which sends the generic point of $\Spec R$ onto the image of the generic point of $\vec s$ -- thus defining two geometric points $\bar \eta$ and $\bar z$. Let also be $\phi_{\bar \eta}$ be an étale path from $\bar \eta$ to $\vec{s}$ as given by change of base point in $\tilde \M_{g,[m]}$. Since étale coverings are proper morphisms, the choice of $T$ defines an étale path $\phi_{\bar \eta \rightsquigarrow \bar z}$ from $\bar\eta$ to $\bar z$, and following Diagram \ref{Diag:InGenTgbp} below:
\begin{equation}\label{Diag:InGenTgbp}
\begin{tikzcd}
I_{\bar z} \arrow{d}{\omega_{\bar z}}  & & I_{\bar \eta}\arrow{d}{\omega_{\bar\eta}} \\
\pi_1(\tilde \M_{g,[m]}, {\bar z}) \arrow{rr}{\phi_{\bar \eta \rightsquigarrow \bar z}} \arrow{rd}{\phi_{\bar z}} &   & \pi_1(\tilde \M_{g,[m]}, {\bar \eta}) \arrow{dl}{\phi_{\bar \eta}} \\
&  \pi_1(\tilde \M_{g,[m]}; \vec{s})
\end{tikzcd}
\end{equation}
with $\phi_{\bar z} = \phi_{\bar \eta}\circ \phi_{\bar \eta \rightsquigarrow \bar z}$.

\medskip

The following lemma is an analog of Grothendieck's specialisation Theorem for the fundamental group in the case of stack inertia groups.

\begin{lem}\label{lem:InSpecGen}
	Let $\eta$ and $z$ be two points in $\tilde\M_{g,[m]}$ such that $z$ is a specialisation of $\eta$. Let $R$ be a valuation ring and $T\in\tilde\M_{g,[m]}(\Spec R)$ whose generic fibre is a geometric point $\bar\eta$ above $\eta$ and whose special fibre is a geometric point $\bar z$ above $z$. Then the choice of $T$ induces an étale path $\bar\eta \rightsquigarrow\bar z$ which sends the stack inertia $I_{\bar \eta}$ into $I_{\bar z}$. 
\end{lem}

\begin{proof} Since étale coverings are proper morphisms, the choice of $T$ defines an étale path $\phi_{\bar \eta \rightsquigarrow \bar z}$ from $\bar\eta$ to $\bar z$ by using the valuative criterion for properness. This choice is by definition compatible with specialisation.
	
	Consider the curves $C_{\bar \eta}$ and $C_{\bar z}$ with their respective automorphism groups $\Aut(C_{\bar \eta})$ and $\Aut(C_{\bar z})$. The stable reduction process induces a morphism $\phi\colon\Aut(C_{\bar \eta})\to \Aut(C_{\bar z})$, where $\phi$ is injective thanks to the non-ramification of the diagonal of $\bar\M_{g,[m]}$. The lemma follows from the commutativity of the diagram
	\[
	\begin{tikzcd}
	\Aut(C_{\bar \eta})=I_{\bar \eta} \arrow[hookrightarrow]{r} \arrow[hookrightarrow]{d}{\phi} & \pi_1({\tilde{\mathcal M}}_{g,[m]},\bar \eta)\arrow{d}{\bar\eta \rightsquigarrow\bar z}\\
	\Aut(C_{\bar z})=I_{\bar z} \arrow[hookrightarrow]{r} & \pi_1(\tilde{\mathcal M}_{g,[m]},\bar z).
	\end{tikzcd}
	\]
\end{proof}

This result should be read in relation with Theorem \ref{theo:GoodGAct} on the generic $G$-deformation of a smooth curve to the boundary of $\tilde{\mathcal M}_{g,[m]}$. Moreover, the subgroups $\phi_{\bar z}(I_{\bar z})$ and $\phi_{\bar\eta}(I_{\bar \eta})$ can be seen seen as subgroups of $\pi_1( \tilde \M_{g,[m]}; \vec{s})$.

\Mysubsubsection\label{subsub:defactcyclo} When the $\Gq$-action $\rho_{\vec s}^I$ on the stack inertia group of the generic point $I_{\bar\eta}$ is given by cyclotomy, one obtains the following description of the tangential $\Gq$-action $\rho_{\vec s}$ on $\pi_1(\tilde{\M}_{g,[m]}\otimes\bar \QQ;\vec s)$ on the stack inertia group $I_{\bar z}$ of the specialisation.

\begin{lem}\label{lem:CyclFromGenToSpec}
	Let $\eta$ be the generic point of an irreducible component of the special loci $\M_{g,[m]}(G)$, and  $z$ a specialisation of $\eta$ in $\tilde{\M}_{g, [m]}$. If there exists a $\Gq$-action $\rho_{\vec s}^I$ on $G< I_{\bar \eta}$ that is given by cyclotomy, then  for $\sigma\in \Gq$ there exists an étale path $\delta_\sigma$ of $\pi_1(\widetilde{\M}_{g,[m]}\otimes \bar \QQ,\vec s)$ such that for any $\gamma \in G< I_{\bar z}$ in  is given by:  
	\begin{equation*}
	\rho_{\vec s}(\sigma).\gamma = \delta_\sigma . \gamma^{\chi(\sigma)} . \delta_\sigma^{-1}
	\end{equation*}
\end{lem}

In the following, we say that such a tangential $\Gal(\bar\QQ/\QQ)$-action on a stack inertia element \emph{is given by $\chi$-conjugacy}.

\begin{proof}
	Let $\gamma$ be generator of $I_{\bar \eta}$ and write  \[
	\tau = \phi_{\bar \eta \rightsquigarrow \bar z}\circ\gamma\circ\phi_{\bar \eta \rightsquigarrow \bar z}^{-1}.,
	\]
        For a $\sigma \in \Gal(\bar \QQ/\QQ)$, the discussion above and the compatibility of $\rho_{\vec s}^I$ and $\rho_{\vec s}$ of Proposition \ref{rem:action_intertielle} give
	\begin{equation*}
	\rho_{\vec s}(\sigma).\gamma = \delta_\sigma . \gamma^{\chi(\sigma)} . \delta_\sigma^{-1}
	\end{equation*}
	where $\delta_\sigma = \sigma(\phi_{\bar \eta \rightsquigarrow \bar z}) \circ \phi_{\bar \eta \rightsquigarrow \bar z}^{-1}$  is an étale path in $\pi_1(\tilde{\M}_{g,[m]}\otimes\bar \QQ;\vec s)$.
\end{proof}

For curves without étale factorisation, the Lemma above and Proposition \ref{prop:GalActTotRam} gives in particular:

\begin{cor}\label{theo:GalActTotRam} Let $\vec s$ be a tangential base point of $\M_{g, [m]}$, denote by $\rho_{\vec s}$ the tangential $\Gq$-representation induced by $\vec s$, and let $G$ be a cyclic stack inertia group of $\M_{g,[m]}$. If $G$ satisfies the non-étale factorization property, then $\rho_{\vec s}$ induces a $\Gal(\bar\QQ/\QQ)$-action on $G$ given by $\chi$-conjugacy.
\end{cor}

\begin{rem}  In $\M_{0,4}\simeq \PP^1\setminus\{0,1,\infty\}$, the \emph{droit chemin} $p$ from the tangential base point $\vec{01}=\Spec \QQ\llbracket t \rrbracket$ to $\vec{10}=\Spec \QQ\llbracket -t \rrbracket$ -- see \cite{IHARA91} -- admits a factorisation by the path $r$ from $\vec{01}$ to $1/2\in \M_{0,4}$. The $\Gq$-action on $p$ gives a factor $f_{\sigma}$ while $r$ gives a factor $g_{\sigma}$ -- see \cite{LS97}. Since the point $1/2\in \M_{0,4}$ represents a point in $\M_{0,[4]}$ with (reduced) cyclic inertia $\ZZ/2\ZZ$, the cocycle $\delta_\sigma = \sigma(\phi_{\bar \eta \rightsquigarrow \bar z}) \circ \phi_{\bar \eta \rightsquigarrow \bar z}^{-1}$ plays a role similar to the factor $g_{\sigma}$.
\end{rem}

\Mysubsubsection \label{subsub:finalProof}
We now establish the main result of the article, which follows from all the results collected in the previous sections: the compatiblity of some local, stack inertia and tangential Galois actions (resp. $\rho_{\bar z}^I$, $\rho_{\vec s}^I$ and $\rho_{\vec s}$ in \S \ref{subsub:defGalAcIn}), the specific action by cyclotomy-conjugacy of \S \ref{subsub:defactcyclo}, the Galois-invariant tangential morphisms of \S \ref{subsub:GqActKnuMor} and the generic degeneracy of $G$-covers of \S \ref{subsub:degGcurve}.

\begin{theo}\label{theo:GalActStrg}
	Let $I$ be a cyclic stack inertia group of $\M_{g,[m]}$. Then the tangential $\Gq$-actions on $\pi_1(\M_{g,[m]}\otimes \bar{\QQ},\vec s)$ are given by $\chi$-conjugacy on $I=\langle\gamma\rangle$:
	\begin{equation*}
		\rho_{\vec s}(\sigma).\gamma = \delta_\sigma\, \gamma^{\chi(\sigma)}\, \delta_\sigma^{-1}
	\end{equation*}
	where $\delta_\sigma$ is an étale path of $\pi_1(\M_{g,[m]}\otimes \bar \QQ,\vec s)$.
\end{theo}

Such a tangential $\Gq$-representation can be explicitly given by a curve of type $X_A$ ($g\geqslant 1$) or $X_B$ ($g\geqslant 2$) of Fig. \ref{Fig:maxDegG1G2}. Recall that the definition of the tangential $\Gq$-action $\rho_{\vec s}$ on $I$ is obtained via a well-defined stack inertia Galois action $\rho_{\vec s}^I$, see Eq.~\eqref{eq:GKI}.

\begin{proof}
        As the case with no étale factorisation is dealt with by Proposition  \ref{prop:GalActTotRam}, the remaining case is when $I$ is the automorphism group of a curve $C\in \mathcal M_{g,[m]}$ with étale factorisation. Since automorphisms of a curves of genus $0$ are without étale factorisation, it can be assumed that $g\geqslant 1$.

	\medskip
	
	 Let $X$ denote the closed point of $\vec S$ and let $X'$ be its partial normalisation at a singular point. Following Proposition \ref{cor:ClutGqComp}, it is then possible to define two tangential $\Gq$-actions $\rho_{\vec s}$ on $\pi_1(\tilde\M_{g,[m]}\otimes \bar \QQ;\vec s) $ and $\rho_{\vec s\, '}$ on $\pi_1(\M_{g-1,[m]+2}\otimes \bar \QQ;\vec s\, ')$ that are compatible with the Knudsen's morphism
	\[
	\begin{tikzcd}
	\pi_1(\M_{g-1,[m]+2}\otimes \bar{\QQ},\vec s\, ') \arrow{r}{\tilde\beta} & \pi_1(\tilde{\M}_{g,[m]}\otimes \bar{\QQ},\vec s).
	\end{tikzcd}
	\]
	Moreover, $\rho_{\vec s}$ and $\rho_{\vec s\, '}$ define two compatible $\Gq$-actions $\rho_{\vec s}^I$ at the level of the stack inertia groups $I_X$ and $I_{X'}$, see Remark \ref{rem:GqIn} \ref{rem:CompActionInertia}. 
	
	Let $\eta$ be the generic point of the special loci $\M_{g,[m]}(I)$, and let $z$ be a specialisation of $\eta$ given by Corollary \ref{Cor:GoodActPoints}. Denoting by $\eta'\in \M_{g-1, [m]+2}(I)$ the generic point of the component containing $\tilde\beta^{-1}(z)$, one obtains a specialisation $\xi=\tilde\beta(\eta')$ of $\eta$ whose normalisation has genus $g-1$ and a $I$-action without étale ramification. The $I$-curves $z\in \bar{\M}_{g,[m]}$ and $\tilde\beta^{-1}(z)\in \tilde{\M}_{g-1,[m]+2}$ give by contraction the curves $C$ in $\bar{\M}_{g,[m]}$ and $\tilde C\in\bar{\M}_{g-1,[m]+2}$ respectively and therefore étale paths $\phi\colon\eta\rightsquigarrow \vec s$ and $\phi'\colon\eta'\rightsquigarrow \vec s\, '$.
	
	 This reduces the description of $\rho_{\vec s}$ on the stack inertia $I$ of ${\M}_{g,[m]}$ to that of $\rho_{\vec s\, '}$ on the stack inertia of ${\M}_{g-1,[m]+2}$: because of the ``without étale factorisation property'', the action $\rho_{\vec s\, '}$ on $I_{\eta'}$ is given by $\chi$-conjugacy in $\M_{g-1,[m]+2}$ by Corollary \ref{theo:GalActTotRam}; by $\beta$-compatibility of the actions and the existence of $\phi$ and $\phi'$, the action $\rho_{\vec s}$ is given by $\chi$-conjugacy on $I_{\xi}$ in $\M_{g,[m]}$. By the property of injectivity under specialisation of Lemma \ref{lem:InSpecGen}, this implies the same property for the $\Gq$-action $\rho_{\vec s}$ on $I$ viewed in the generic automorphism group $I=I_\eta < I_{\xi}$.

	Since the canonical morphism $\pi_1(\M_{g,[m]}\otimes \bar \QQ) \to \pi_1(\tilde \M_{g,[m]}\otimes\bar{\QQ})$ induces an injection at the level of stack inertias, the $\Gq$-action by $\chi$-cyclotomy on $I$ viewed as inertia group in $\pi_1(\tilde \M_{g,[m]}\otimes \bar \QQ)$ then finally implies the same for $I$ as inertia group in $\pi_1(\M_{g,[m]}\otimes \bar \QQ)$.
\end{proof}

\Mysubsubsection \label{subsubsec:opening} 
	By analogy with the Deligne-Mumford stratification, results and methods of this article encourage to lead further studies of the arithmetic of the stack inertia stratification (see \cite{DOUAI06} for a description), either by describing the Galois action for higher non-cyclic stack inertia strata, or by describing the conjugacy factors $\delta_\sigma$ in the $\chi$-conjugacy action of Theorem \ref{theo:GalActStrg}.
	
	\medskip
	
	For the Deligne-Mumford stratification, Grothendieck-Murre theory implies that the tangential $\Gq$-actions on $\pi_1^{et}(\M_{g, [m]}\otimes\bar{\QQ})$ is given by $\chi$-conjugacy on the divisorial inertia groups $I_D,\ D\in \partial \bar{\M}_{g, [m]}$, while the \emph{stratification by topological type $(g,m)$} -- given by the Knudsen morphisms -- reduces the description of these actions to the $4$ stratas of modular dimension $1$ and $2$ only. A finer description of the conjugacy factor is obtained by comparing $\Gq$-actions on different topological strata via Knudsen clutching morphisms, see for example \cite{NAK96} in the case of the clutching $\bar\M_{g_1,m_1}\times \bar\M_{g_2,m_2}\to \bar\M_{g_1+g_2,m_1+m_2-2}$.
	
	The stack inertia stratification is given by the decreasing dimensional locus $\M_{g,[m]}(G_i)\subset \M_{g,[m]}(G_{i-1})$, where $G_i> G_{i+1}$ and $G_0=\{Id\}$, see \cite{LMB00} Theorem 11.5. We show that \emph{the cyclic stack inertia stratification is given by the branch data $\bf kr$ of \S \ref{subsub:GDefGenTypeMarked}}, thus the corresponding strata $\M_{g,[m],{\bf kr}}(\ZZ/n\ZZ)$; in \S \ref{subsub:finalProof}, we compare two stack inertia $\Gq$-actions 
	\[
		\rho_{\vec s,x}^I\colon \Gq\to \Aut[I_x]\quad \text{and} \quad \rho_{\vec s\, ',y}^I\colon \Gq\to \Aut[I_y]
	\]
	on the automorphism groups $I_x\simeq I_y$ of curves $x\in \M_{g,[m],{\bf kr}}(\ZZ/n\ZZ)$ and $y\in \M_{g-1,[m]+2,{\bf kr'}}(\ZZ/n\ZZ)$ \emph{of stack strata of different types} -- i.e with distinct $\ZZ/n\ZZ$-invariants (genus of the quotients $g_x'\neq g_y'$, and branch data ${\bf kr}\neq{\bf kr'}$). This process deserves the name of \emph{inertial limit Galois action}. Developing a combinatorial description of the geometry of the cyclic stack inertia stratification should lead to a finer description of these inertial limit Galois actions, for example by \emph{comparing the conjugacy factors of the prime and general cyclic stratas}.
	
	\bigskip
	
	In another direction, and following a long Geometric Galois Action tradition, this $\chi$-conjugacy Galois action also motivates the search of \emph{new stack inertia $\Gq$-equations in higher genus} -- see \cite{NAKTSU03} and \cite{SCH06} in genus $0$.

\vfill

\section*{Acknowledgements}
B.~Collas thanks Prof.~H.~Nakamura for his encouragements at the beginning of the development of this manuscript. The first author has benefited from many fruitful environments during the preparation of this paper, and accordingly thanks Prof.~Moshe Jarden and \emph{Tel Aviv University}, Prof.~Hossein Movasati and the \emph{IMPA} for their hospitality.

\newpage
\bigskip
%

\providecommand{\bysame}{\leavevmode\hbox to3em{\hrulefill}\thinspace}

\begin{center}
	$\star$\\
	$\star\hphantom{\star}\star$\\
	
	\vspace*{\fill}
\end{center}
\begin{multicols}{2} 
	\centering
	\small
	\textsc{Benjamin Collas}\\
	-- --\\
	
	\textsc{Mathematisches Institut}\\
	\textsc{ Westfalische Wilhelms-Universität Münster}\\
	Einsteinstr. 62,\\
	D-48149 Münster (Deutschland)\\ 
	\medskip
	\emph{email :} collas@math.cnrs.fr
	
	\columnbreak
	
	\textsc{Sylvain Maugeais}\\
	-- --\\
	\textsc{Laboratoire manceau de Mathématiques}\\
	\textsc{Université du Maine}\\
	Av. Olivier Messiaen, BP 535 \\
	72017 Le Mans Cedex (France)\\
	\medskip
	\emph{email :} sylvain.maugeais@univ-lemans.fr
\end{multicols}	
\end{document}